\documentclass{amsart}
\usepackage[mathscr]{eucal}
\usepackage{color}
\usepackage{xypic}
\usepackage{amsfonts}   
\usepackage{amsmath}
\usepackage{amsthm}
\usepackage{amssymb}
\usepackage{latexsym}
\usepackage[english]{babel}
\usepackage[utf8]{inputenc}

\let\nc\newcommand

\nc{\la}{\label}

\newtheorem{conjecture}{Conjecture}
\newtheorem{theorem}{Theorem}[section]
\newtheorem{definition}[theorem]{Definition}
\newtheorem{corollary}[theorem]{Corollary}

\newtheorem{lemma}[theorem]{Lemma}
\newtheorem{proposition}[theorem]{Proposition}
\newtheorem{example}[theorem]{Example}
\newtheorem{remark}[theorem]{Remark}

\def\k{\mathsf k}
\def\C{\mathbb C}

\newcommand{\Frac}{{\rm{Frac}}}

\newcommand{\Aut}{{\rm{Aut}_{\c}}}

\title{Holonomic modules for rings of invariant differential operators}
\author{ Vyacheslav Futorny and Jo\~ao Schwarz}

\AtEndDocument{\bigskip{\footnotesize%
 
 \addvspace{\medskipamount}
 (V.\,Futorny) \textsc{Instituto de Matematica e Estatistica, Universidade de S\~{a}o Paulo, Caixa Postal 66281,\\ S\~{a}o Paulo, CEP 05315-970, Brasil and International Center for Mathematics, SUSTech, Shenzhen, China} \par
  \textit{E-mail address}: \texttt{futorny@ime.usp.br}  \par
   (J.\,Schwarz) \textsc{Instituto de Matematica e Estatistica, Universidade de S\~{a}o Paulo, Caixa Postal 66281,\\ S\~{a}o Paulo, CEP 05315-970, Brasil} \par
  \textit{E-mail address}: \texttt{jfschwarz.0791@gmail.com }  \par
 
   \addvspace{\medskipamount}
  
}}

\begin{document}
\maketitle


\begin{abstract}
We study holonomic modules for the rings of invariant  differential operators  on affine commutative domains with finite Krull dimension with respect to arbitrary actions of finite groups.
 We prove the Bernstein inequality for these rings. Our main tool is  the filter dimension introduced by Bavula. 
  We extend the results for the invariants of the Weyl algebra with respect to the symplectic action of a finite group,  for the rings of invariant differential operators on quotient varieties, and invariants of certain generalized Weyl algebras under the linear actions.
We show that the filter dimension  of all above mentioned algebras equals $1$.

\medskip

\noindent {\bf Keywords: Filter dimension, holonomic modules, generalized Weyl algebras, invariant differential operators} 
\medskip

\noindent {\bf 2020 Mathematics Subject Classification:  16P90, 16S32, 16D30, 16W70
}
\medskip

\end{abstract}

\section{Introduction}
Let $\k$ be the base field. All rings in the paper are $\k$-algebras. All modules are left modules, unless said otherwise.

In this paper we address representations of   subrings of invariants of generalized Weyl algebras. The later class of algebras was introduced and studied in \cite{B0}. Many important algebras of small Gelfand-Kirillov dimension arising in noncommutative geometry  are generalized Weyl algebras, such as the first Weyl algebra and its quantization; the quantum plane and the quantum sphere; $U(sl_2(\k))$ and its quantization; the Heisenberg algebra and its quantizations; quantum $2 \times 2$ matrices; Witten's and Woronowic's deformations; Noetherian down-up algebras (cf. \cite{BY}). For the representation theory of generalized Weyl algebras we refer to \cite{B0}, \cite{B2}, \cite{BX}, \cite{B5}, \cite{B6}.

We will consider a category of holonomic modules for certain  subrings of invariants of generalized Weyl algebras under the action of  finite groups.

For an arbitrary algebra $A$ denote by $GK(A)$ the Gelfand-Kirillov dimension of $A$. Then 
a finitely generated $A$-module $M$ is called \emph{holonomic} (cf. \cite{Krause}, \cite{McConnell}) if $$GK(M) =\frac{1}{2} GK (A/Ann(M)).$$ 

Assume $char \, k =0$ and consider the $n$-th Weyl algebra $A_n(\k)$. Then $A_n(\k)$ is isomorphic to  the ring of differential operators on the polynomial algebra $\k[x_1,\ldots,x_n]$, or equivalently on the affine space $\mathbb A^n$. It has
   generators $x_1,\ldots, x_n, y_1,\ldots, y_n$ and defining relations 
 $$[x_i,x_j]=[y_i,y_j] = 0; [y_i,x_j]= \delta_{ij},$$ $i,j=1,\ldots, n$. Note that $GK (A_n(\k))= 2n$ and  for every finitely generated $A_n(\k)$-module $M$ holds 
  {\it the  Bernstein inequality:}  $GK (M) \geq n$ \cite{Bernstein}. 
  Holonomic $A_n(\k)$-modules
   are exactly the modules of the minimal Gelfand-Kirillov dimension, they constitute an important abelian subcategory of modules for the Weyl algebra (\cite{Borel}).  The Bernstein inequality holds also for the rings of differential operators on smooth affine varieties (\cite{Borel}; cf. Section 3 bellow).
   
Assume $\k$ algebraically closed  of characteristic 0 and let $\mathfrak{g}$ be an algebraic Lie algebra of finite dimension. Then  {\it the Gabber's inequality} $$GK \, (U(\mathfrak{g})/Ann(M)) \leq 2GK (M)$$ holds for any
 finitely generated $\mathfrak{g}$-module $M$   (cf. \cite{Jantzen}). The holonomic $\mathfrak{g}$-modules are those modules with the minimal Gelfand-Kirillov dimension
 (cf. \cite{Krause}, Chapter 9). However, if $\mathfrak{g}$ is not algebraic then it is known that the Gabber's inequality does not hold \cite{McConnell2}.

To study holonomic modules and analogues of the Bernstein inequality for infinite-dimensional affine simple algebras over an arbitrary field, Bavula \cite{B2} introduced the notion of the {\it filter dimension}, denoted here by $fdim$ 
 (see also \cite{B3} for details and applications). 
 Our first  main result gives the
  filter dimension of certain algebras of invariant differential operators with respect to finite groups action.
 
\begin{theorem}\label{thm-main}  The filter dimension equals $1$
for the following algebras:
\begin{itemize}
\item $\mathcal{D}(A)^G$, where $A$ is an affine regular commutative domain with the finite Krull dimension and $G$ is a finite group of automorphisms of $A$;
\item $A_n(\k)^G$, where $G$  a finite group of symplectic automorphisms of the Weyl algebra $A_n(\k)$; 
\item $\mathcal{D}(\mathbb{A}^n/G)$, $\k$ is algebraically closed and 
$G$ is a finite group of linear automorphisms of the affine space $\mathbb{A}_{\k}^n$.
\end{itemize}
\end{theorem}

We also extend the results of \cite{B4} for the invariants of generalized Weyl algebras under suitable actions of complex reflection groups of type $G(m,p,r)$. Namely, we have

\begin{theorem}\label{thm-main2}
Let $D(a, \sigma)$ be a generalized Weyl algebra rank $r$ of pure type and 
 $G=G(m,p,r)$. Then $fdim \, D(a, \sigma)^G = 1$.
\end{theorem}

 With the idea of the filter dimension we have an alternative definition of holonomic modules.

\begin{definition}\label{def1}
Let $A$ be an affine infinite-dimensional simple algebra. The infimum $h_A$ of the set 
$$\{ GK(M) | M \mbox{is finitely generated} \, A\mbox{-module} \}$$ is called the holonomic number of $A$.
\end{definition}

 Suppose that $A$ is simple, $fdim \, A \geq 1$ and the infimum $h_A$ in Definition \ref{def1}  is actually the minimum.
 Then in this case the definition of a holonomic $A$-module above can be replaced  by the following: a finitely generated $A$-module $M$ is holonomic if 
  $GK (M)= h_A$ (cf. Theorem \ref{BavBer}).
   The situation though is unclear if $fdim \, A <1$. 


It was shown in \cite{B4}  that, given a finite Coxeter group $W$ action on the Weyl algebra $A_n(\k)$, the Bernstein inequality holds for $A_n(\k)^W$: for every finitely generated $A_n(\k)^W$-module $M$, $GK (M) \geq n$.
 We generalize this result and prove the Bernstein inequality for  $A_n(\k)^G$ with linear action of an arbitrary  finite group $G$.  We note that our approach is different from the one in \cite{B4} (cf. Theorem \ref{Bernsteinnew} bellow). 
 We also prove a similar result for more general actions of finite groups of symplectic automorphisms on $A_n(\k)$. 
Moreover, we extend this result to the ring of invariant  differential operators $\mathcal{D}(A)^G$ on an arbitrary affine regular commutative domain $A$ with finite Krull dimension  and any finite group $G$.  Our approach relies on the computation of the filter dimension in order to illustrate its application. In particular,  all  algebras mentioned in Theorem \ref{thm-main} and Theorem \ref{thm-main2}
 have a "good" theory of holonomic modules:

\begin{theorem}\label{thm-main3} Let $A$ be an affine regular commutative domain with finite Krull dimension and $GK (A)=n$, and let 
  $B$ be  one of the algebras from Theorem \ref{thm-main} or Theorem \ref{thm-main2}.
 Then every finitely generated holonomic $B$-module  is a cyclic torsion module of finite length  and  their holonomic number is $n$.
\end{theorem}

We apply developed technique to compute the   filtered dimension and the     Krull dimension of rational Cherednik 
algebras. In particular, we show 

\begin{theorem}\label{thm-main4}  
The filtered dimension of generic rational Cherednik algebras and their spherical subalgebras equals $1$. 
\end{theorem}




The structure of the paper is as follows. 
In Section \ref{sec-filter}, we recall the notion of a filter dimension and its main properties. Section \ref{sec-holonomic} is the technical core of the paper: we prove a number of results on holonomic modules for simple somewhat commutative algebras, and simple filtered semi-commutative algebras which are Noetherian but not Artinian, with particular emphasis on certain generalized Weyl algebras. All of these algebras are examples of so-called \emph{algebras with multiplicity}.
Some of these results are probably known to specialists but we included them for the clarity of exposition. In Section \ref{invariants} we discuss the filter dimension of rings of invariants under the action of a finite group. 
In Section \ref{sec-hol-inv} we consider the invariants of Weyl algebras, rings of invariant differential operators on quotient varieties, and invariants of generalized Weyl algebras.
We show that these algebras have a nice category of holonomic modules (cf. Theorems \ref{Bernsteinnew},  \ref{quotient1}, \ref{genWeyl}, \ref{goodcategory2}). Finally, in Section \ref{section-cherednik} we compute the filtered dimension and the Krull dimension of the rational Cherednik algebras and its spherical subalgebras.

\section{The filter dimension}\label{sec-filter}
Let $A$ be an algebra over $\k$. Every finite-dimensional subspace $V\subset A$  containing the identity will be called a \emph{frame} of $A$. Let us recall the notion of the Gelfand-Kirillov (GK) dimension of algebras and modules (cf. \cite{Krause}).

\begin{definition}\[ GK (A) = sup_V \, limsup_{n \mapsto \infty} \frac{log \, dim \, V^n}{log \, n}, \] where $V$ ranges through all frames of $A$.
Let $M$ be an $A$-module. Then

\[ GK (M) = sup_{V,F} \, limsup_{n \mapsto \infty} \frac{log \, dim \, V^nF}{log \, n}, \]

where $F$ runs through all finite dimensional subspaces of $M$ and $V$ runs through all frames of $A$.
\end{definition}

 Recall (\cite[8.1.9]{McConnell}) that a \emph{finite-dimensional filtration} of an algebra $A$ is a filtration $\mathcal{F}=\{A_i\}_{i \geq 0}$ such that $A_0 = \k$ and $dim \, A_i < \infty, i>0$. If $N$ is an $A$-module with filtration $\{ N_i \}_{i \geq 0}$, then the filtration is finite-dimensional if $dim \, N_i < \infty$, $i \geq 0$.
Suppose that $A$ is an infinite-dimensional affine algebra generated by $a_1, \ldots a_n$.
Define a finite-dimensional filtration $\mathcal{F}=\{A_i\}_{i \geq 0}$ in the following way: $A_0= \k, \, A_1 = span \langle 1, a_1, \ldots, a_n \rangle, A_i=A_1^i$. Now let $M= A M_0$ be a finitely generated $A$-module with finite-dimensional generating space $M_0$. 
Then $\mathcal{F}$ induces a natural filtration $\Omega = \{M_i \}_{i \geq 0}$ on $M$, where $M_i=A_i M_0$. 

\emph{The return function} $f_{\mathcal{F},M_0}: \mathbb{N} \rightarrow \mathbb{N} \cup \{ \infty \}$ is defined as follows:
 $f_{\mathcal{F},M_0}(i) =$

\[ min \{ j \in \mathbb{N} \cup \{ \infty \} : A_j M_{i,g} \supset M_0 \, \forall M_{i,g} \}, \]
where $M_{i,g}\subset M_i$ runs over the finite-dimensional  subspaces  which generate $M$ as a module.

For each function $f: \mathbb{N} \rightarrow \mathbb{N} \cup \{ \infty \}$ define $\Gamma(f)$ as

\[ inf\{r \in \mathbb{R}|f(i) \leq i^r, i>>0 \}. \]

Note that, if $f(i)=p(i)$ for all sufficiently large $i$ for some polynomial $p \in \mathbb{Q}[x]$,   then $\Gamma(f)= deg \, p$ (\cite{Krause}, Chp. 2).

For a finitely generated $A$-module $M$, the \emph{filter dimension} is defined as (\cite{B2})
$$fdim(M)=\Gamma(f_{\mathcal{F},M_0}).$$  It
 is independent of the choice of the generating set $a_1,\ldots, a_n$ and the  subspace $M_0$ (\cite[Lemma 1.1]{B2}). The filter dimension of $A$, $fdim(A)$, is its filter dimension as an $A \otimes_\k A^{op}$-module.

\begin{example}[\cite{Bnew1}]\label{example}
Let $B$ be an affine regular commutative algebra with finite Krull dimension. Then the filter dimension of the ring of differential operators $\mathcal{D}(B)$ on $B$ equals  $1$.
\end{example}

 The following analog of the Bernstein inequality was shown in \cite{B2}:

\begin{theorem}\label{BavBer}
Let $A$ be an  infinite-dimensional affine simple algebra and $M$ a non-zero finitely generated module. Then

\[ GK (M) \geq \frac{GK (A)}{1+max(1+fdim(A))}.\]
\end{theorem}

In view of Example \ref{example}, this reproves the usual Bernstein inequality for rings of differential operators on smooth affine varieties (cf. \cite{Borel}).

In general, a  computation of the filter dimension of an algebra is a difficult problem  (\cite{B3}, \cite{Krause} 12.9). However, it has an important property to be Morita invariant. Namely:

\begin{theorem}[\cite{B4},Theorem 1.3, Theorem 1.6]\label{morita}
Morita equivalent infinite-dimensional affine simple algebras have the same filter dimension and the same holonomic number.
\end{theorem}

\section{Holonomic modules for algebras with multiplicity}\label{sec-holonomic}
\subsection{Algebras with multiplicity}
From now on we assume that $char \, \k = 0$.
Let $A$ be a finitely generated infinite-dimensional  simple Noetherian  algebra, with a finite-dimensional filtration $\mathcal{F}=\{A_i \}_{i \geq 0}$ such that $gr_\mathcal{F} \, A$ is finitely generated Noetherian.

Let $M$ be a finitely generated module over $A$ with a filtration $\Omega=\{M_i \}_{i \geq 0}$. 
We assume that $\Omega$ is a \emph{good filtration}, that is
 $gr_\Omega M$ is a finitely generated module over $gr_\mathcal{F} \, A$. Recall that every finitely generated module has a good filtration (\cite{Krause}, Ch. 6).

Consider a short exact sequence of modules $0 \rightarrow M' \rightarrow M \rightarrow M'' \rightarrow 0$ and the filtrations $\Omega'=\{M_i'=M_i \cap M'\}_{i \geq 0}$  and $\Omega''=\{M_i''=M_i +M/M\}_{i\geq 0}$ of $M'$ and $M''$ respectively.
This induces an exact sequence of $gr_\mathcal{F} \, A$-modules:

\[ 0 \rightarrow gr_{\Omega'} M' \rightarrow gr_\Omega M \rightarrow gr_{\Omega''} M'' \rightarrow 0. \]

Since $ gr_\Omega M$ is  finitely generated  and $gr_\mathcal{F} \, A$ is Noetherian then $\Omega'$ and $\Omega''$ are  good filtrations.

\begin{definition}\label{{theoremB11}} \cite[12.6]{Krause}
We will say that $A$ is an \emph{algebra with multiplicity} if for every finitely generated $A$-module $0 \neq M$, $GK (M)$ is a non-negative integer, and there is a function $M\mapsto e(M)$, 
where $e(M)$ is a non-negative integer, called \emph{the multiplicity} of $M$ such that:
 for a short exact sequence of finitely generated $A$-modules
 $$0 \rightarrow M' \rightarrow M \rightarrow M'' \rightarrow 0$$ we have $GK (M) = max \{ GK (M'), GK (M'') \}$, and if $GK (M) = GK (M') = GK (M'')$, then $e(M)=e(M')+e(M'')$.

\end{definition}

Since the Gelfand-Kirillov dimension of finitely generated modules are non-negative integers, there exists a non-zero module whose Gelfand-Kirillov dimension is precisely the holonomic number, which is a non-negative integer. Also note that if $M$ is any non-zero finitely generated module then $GK (M) \geq 1$ and $ e(M) \geq 1$, as  $A$ is simple infinite-dimensional.

\subsection{Somewhat commutative algebras}
A \emph{somewhat commutative algebra} $A$ is an algebra with a finite-dimensional filtration 
$\mathcal{F}=\{A_i \}_{i \geq 0}$ such that the graded associated algebra is finitely generated 
commutative  (\cite[8.6.9]{McConnell}). Affine somewhat commutative algebras are  Noetherian. Then we have from 
\cite[Corol. 8.6.20]{McConnell}:

\begin{corollary}
Any simple somewhat commutative algebra is an algebra with multiplicity.
\end{corollary}


\begin{example}\label{example23}
The Weyl algebra $A_n(\k)$ with the Bernstein filtration  is a simple somewhat commutative algebra.
\end{example}

\begin{example}\label{example21}
The ring of differential operators $\mathcal{D}(X)$ on an affine smooth algebraic variety $X$ is a simple somewhat commutative algebra; or, more generally, $\mathcal{D}(B)$ the ring of differential operators on an affine regular commutative domain (cf. [15.1.21, 15.3.7, 15.5.6]\cite{McConnell}, \cite[Section 5]{B3}).
\end{example}

\subsection{Filtered semi-commutative algebras}

A \emph{filtered semi-commutative algebra} $A$ is a finitely generated algebra with a finite-dimensional filtration $\mathcal{F}=\{A_i \}_{i \geq 0}$ such that the graded associated algebra is semi-commutative (\cite{McConnell3}) i.e., generated by elements $x_1, \ldots, x_n$ with relations $x_i x_j = \lambda_{ij} x_j x_i$, for some $0 \neq \lambda_{ij} \in \k$.

\begin{example}
Examples of such algebras include $\mathcal{O}_q(\mathbb{M}_n(\k))$, $\mathcal{O}_q(GL_n(\k))$, \\$\mathcal{O}_q(SL_n(\k))$ and $U_q(sl_n)$, where $0 \neq q$ is not a root of unity (\cite{McConnell3}).
\end{example}

We have 

\begin{proposition}\cite[Theorem 3.8]{McConnell3}
Any simple filtered semi-commutative algebra is an algebra with multiplicity.
\end{proposition}

\subsection{Generalized Weyl algebras}
We now introduce our main example of algebras with multiplicities.

Let $D$ be an algebra over $\k$. Let $a=(a_1, \ldots, a_n)$ be a n-uple of elements of $Z(D)$. Let $\sigma=(\sigma_1, \ldots, \sigma_n)$ be a tuple of automorphisms of $D$ such that $\sigma_i \sigma_j = \sigma_j \sigma_i$, $\sigma_i(a_j)= a_j$, if $j \neq i$. The \emph{generalized Weyl algebra} (GWA for short) of rank $n$ (\cite{B0}) is the algebra with generators $D, X_i, Y_i$, $i=1, \ldots, n$, and relations 

\[ X_i \lambda = \sigma_i(\lambda) X_i; Y_i \lambda = \sigma_i^{-1}(\lambda) Y_i, \lambda \in D; \]
\[ Y_iX_i = a_i, X_i Y_i = \sigma_i(a_i) \].

We will denote this algebra by $D(a, \sigma)$ and 
 call  $X_i, Y_i$'s the GWA generators,  $D$   the defining algebra. We will assume that $D$ is an affine commutative domain. In this case
 every generalized Weyl algebra is a Noetherian domain and, hence, an Ore domain.
The tensor product over $\k$ of two generalized Weyl algebras $D(a,\sigma) \otimes D(a',\sigma') \simeq (D \otimes D')(a\ast a', \sigma \ast \sigma'), $ is again a generalized Weyl algebra,  where $\ast$ is the tensor product of automorphisms, and the concatenation of $a,a'$. In particular we have

\begin{definition}[\cite{B0}]
Let $D(a, \sigma)$ be a GWA of rank $1$, with automorphism  $\sigma$ of infinite order. We define $D_n(a, \sigma)$ to be $D(a, \sigma)^{\otimes n}$. It is itself a GWA $D'(a=(a_1, \ldots, a_n), \sigma=( \sigma_1, \ldots, \sigma_n))$, where $D' = D \otimes \ldots \otimes D$ n times, $a_i = 1 \otimes\ldots \otimes a \otimes \ldots 1$, $a$ in the i-th position, and $\sigma_i=1 \otimes\ldots \otimes \sigma \otimes \ldots 1$, $\sigma$ in the i-th position.
\end{definition}

Following \cite{B5} we introduce generalized Weyl algebras of classical and of quantum types.
 Let $A=\k[h](a, \sigma)$ be a generalized Weyl algebra of rank 1 with $\sigma(h)=h-1$.
Assume that there is no  irreducible polynomial $p \in \k[x]$ such that both $p, \sigma^i(p)$ are multiples of $a$ for 
any $i \geq 0$. Then $A$ is a generalized Weyl algebra of rank 1  of \emph{simple classical type}.  Examples of such algebras are the Weyl algebras, certain subalgebras of their invariants, and certain  primitive quotients of $U(sl_2)$, among others (cf. \cite[Section 2]{B5}).

If  $A=\k[h^\pm](a, \sigma)$ is a generalized Weyl algebra of rank $1$ with $\sigma(h)=\lambda h$, $0,1 \neq \lambda \in \k$,  $a \in \k[h]$,  $\lambda$ is not a root of unity and  there are no irreducible polynomials $p \in \k[x]$ such that both $p, \sigma^i(p)$ are multiples of $a$ for any $i \geq 0$, then $A$ is of \emph{simple quantum type}. Examples of such algebras are certain primitive quotients of $U_q(sl_2)$, the quantum torus (cf. \cite[Section 2]{B5}).

Both kinds of algebras are simple.

Let $A$ be a generalized Weyl algebra of rank 1 of simple classical or quantum type, $X,Y$  the $GWA$ generators, i.e., $XY=h$. Let $m = deg \, a$ and $\alpha$ its leading coefficient. Define a finite-dimensional filtration by $\mathcal{B}_A = \{B_i \}_{i \geq 0}$, $B_n = span \langle H^i v_j \rangle, 2 |i| + m |j| \leq n$, where $v_j = X^j$ if $j >0$, $v_j =Y^j$ if $j<0$, $v_0=1$. Let $A = \otimes_{j=0}^s A_j$ be a generalized Weyl algebra of rank $s$, where $A_j$ are generalized Weyl algebras of simple classical or quantum types.  Consider the tensor product of the filtrations $\mathcal{B}_A = \otimes_{j=0}^s \mathcal{B}_{A_j}$.   Then $\mathcal{B}_A$ is a finite-dimensional filtration.

\begin{definition}\label{mixed}
A generalized Weyl algebra $A$  of rank $s$ which is the tensor product of generalized Weyl algebras of simple quantum and classical type is called a generalized Weyl algebra of mixed type. If all factors of the tensor product are of the same simple type then  the algebra $A$ is called a generalized Weyl algebra of pure type. They are simple algebras.
\end{definition}

We have
\begin{theorem}\label{GWA}
Let $A$ be a  generalized Weyl algebra of rank $s$. 
\begin{itemize}
\item
If $A$ if mixed type then $GK (A)=2s$. 
In this case for each finitely generated $A$-module $M$ we have $GK (M) \geq s$. 
\item If $A$ is  of pure type then $A$ is algebra with multiplicity.
\end{itemize}
\end{theorem}

\begin{proof}
 The  claim about generalized Weyl algebras   of mixed type  follows from \cite[Theorem 2.1]{B5}. For generalized Weyl algebras of pure type, the filtration  above shows that it is either somewhat commutative (the case of all tensor factors of simple classical type); or else filtered semi-commutative after taking a quotient by normal elements in the graded associated algebra (the case of all tensor factors  of simple quantum type), as shown in \cite[Section 2]{B5} (cf. \cite[Theorem 2.2, 3.2]{B1}).
\end{proof}

\subsection{Category $\mathcal{H}$}

For an algebra with multiplicity $A$, denote
 by $\mathcal{H}=\mathcal{H}(A)$ the category of finitely generated holonomic modules, that is  for any $M\in \mathcal{H}$, $GK (M)= h_A$.


\begin{theorem}\label{abelian1}
 The category $\mathcal{H}$ is  abelian and every module $M\in \mathcal{H}$ has finite length bounded by $e(M)$. Finally, if $A$ is not Artinian, every holonomic module is cyclic.
\end{theorem}

\begin{proof}
First statement is clear. Let $M \in \mathcal{H}$,  $M = M_0 \supset M_1 \ldots \supset M_s$ be a strictly descending chain of submodules. Hence $M_i \in \mathcal{H}$ and $M_i/M_{i+1} \in \mathcal{H}$. As $e(M_i/M_{i+1}) \geq 1$, then we have
 $$s \leq \sum_{j=0}^{s-1} e(M_j/M_{j+1}) =  e(M/M_s) \leq e(M).$$ 
 Hence, every module $M \in \mathcal{H}$ is Artinian, and since $M$ is finitely generated then $M$ is also Noetherian. Finally, if $A$ is not Artinian, $M$ is cyclic by \cite[Theorem 2.5]{Coutinho}.
\end{proof}




We immediately have

\begin{corollary}
Every module $M \in \mathcal{H}$  with $e(M)=1$ is irreducible.
\end{corollary}


We also have

\begin{corollary}\label{prop-torsion}
If $GK (A) > h_A$, then every module in $\mathcal{H}$ has torsion.
\end{corollary}

\begin{proof}
Let $M \in \mathcal{H}$, $ 0 \neq m \in M$. Consider the map $\phi: A \rightarrow M$, $a \mapsto am$. Then $Im \, \phi$ is a non-null submodule of $M$, and hence by Theorem \ref{abelian1}, belongs to $\mathcal{H}$. Consider the short exact sequence $0 \rightarrow ker \, \phi \rightarrow A \rightarrow Im \, \phi \rightarrow 0$. $GK (A) = max \{ GK (ker \, \phi), GK (Im  \, \phi) \}$, so $GK (ker \, \phi) = GK (A) >0$
as $GK \, ker (\, \phi) = h_A$
; this implies that the kernel is nonzero and hence  $M$ has a torsion.
\end{proof}

On the other hand the following is valid.

\begin{proposition}
Let $A$ be  a domain with multiplicity and $I$ a non-zero left ideal of $A$. Then $GK (A/I) \leq GK (A) -1$.
\end{proposition}

\begin{proof}
Suppose first that $I=Aa$, $0 \neq a \in A$. Consider a short exact sequence 
$$0 \rightarrow A \rightarrow A \rightarrow A/Aa \rightarrow 0,$$ where the first map is a multiplication by $a$. If $GK (A/Aa) = GK (A)$, then $m(A) = m(A) + m(A/Aa) > m(A)$ which is clearly absurd. Hence, $GK (A/Aa) < GK (A)$. In the general case, $I$ contains a principal left ideal $Aa$, and $A/I$ is a quotient of $A/Aa$. Hence $GK (A/I) < GK (A)$.
\end{proof}


\begin{theorem}\label{curious} Let $A$ be  a domain with multiplicity such that
 $GK (A) = 2$ and $h_A=1$. Then a finitely generated $A$-module  $M$ belongs to $ \mathcal{H}$ if and only if it has a torsion.
 \end{theorem}

\begin{proof}
Assume that  the module $M$ has a torsion. Let us show that $M$ is holonomic. Suppose that $M$ is generated by elements $m_1, \ldots, m_s$. Since $M$ has a torsion, $Am_i$ is a quotient of $A/J_i$ for certain left ideals $0 \neq J_i \subset A, i=1,\ldots, s$. Each $A/ J_i$ is holonomic by the proposition above with the Gelfand-Kirillov dimension  $1$.   Hence all modules $Am_i$ and  $M = \sum_{i=0}^s A m_i$ are  holonomic  by Theorem \ref{abelian1}. 
\end{proof}

As a consequence we immediately obtain the following well-known result

\begin{corollary} Let $\k$ be algebraically closed.
Every finitely generated weight module for $A_n(\k)$ is holonomic.
\end{corollary}

\begin{proof} Every finitely generated weight $A_n(\k)$-module $M$ has finite length. Since $A_n(\k)$ is an algebra with multiplicity, it suffices to consider the case when $M$ is a simple module. In this case we have that $M$ is the tensor product of $n$-simple weight modules for $A_1(\k)$ (cf. \cite{BBF}).
 But every such $A_1(\k)$-module has a torsion. Hence, the result follows from Theorem \ref{curious}.  \end{proof}

\section{Filter dimension of invariants}\label{invariants}
Assume that $char \, \k =0$ and
 $A$ is a simple finitely generated Noetherian algebra. Let $a_1, \ldots a_n$ be the generated of $A$.
 
 Define a filtration $\mathcal{F}=\{A_i\}_{i \geq 0}$: $A_0= \k, \, A_1 = span \langle 1, a_1, \ldots, a_n \rangle, A_i=A_1^i$, $i=1, \ldots, n$.
Set $$\nu_\mathcal{F}(i) = inf \{j \in \mathbb{N} \cup \{ \infty \} | 1 \in A_j a A_j, \, \forall a \in A_i \}.$$

\begin{proposition}[\cite{B5}, Lemma 1.1]
$\Gamma(\nu_\mathcal{F}) = fdim \, A$.
\end{proposition}

We recall some facts on the invariants of noncommutative rings that will be  used in what follows (cf. \cite[Theorem 2.5, Corollary 2.6]{Montgomery},  \cite{Montgomery2} and \cite[8.2.9]{McConnell}).

\begin{theorem}\label{ring}
Let $G$ be a finite group of outer automorphisms of $A$ and $A*G$ the skew group ring.
\begin{enumerate}

\item 
$A^G$ is a simple ring Morita equivalent to $A*G$ and  $GK (A^G)= GK (A)$.
\item
$A^G$ is finitely generated and Noetherian; if $A$ is not Artinian then neither is $A^G$. 
\end{enumerate}
\end{theorem}


We are now going to explore the connection between the filter dimensions of $A$ and $A^G$. Suppose that $G$ acts by outer automorphisms.  We can assume without loss of generality that $G$ stabilizes $A_1$: just consider the set $\{ g(a_i)| g \in G, i=1,\ldots, n \}$ as generators for $A$. Then the algebra $A*G$ is simple by Theorem \ref{ring}, and posseses a filtration $\mathcal{F}'=\{ B_i \}_{i \geq 0}$, with $B_0 = \k$ and $B_1 = span \langle a_1, \ldots, a_n, g \in G \rangle$.

\begin{lemma}
For every $i$, $\nu_{\mathcal{F}'}(i) \leq \nu_\mathcal{F}(i)$.
\end{lemma}

\begin{proof}
Let $b \in B_i$. If $\nu_\mathcal{F}(i) = \infty$ then the claim is clear. Otherwise, consider the idempotent $e= \frac{1}{|G|} \sum_{g \in G} g$. Then $e \in B_j, j \geq 1$. We also have $A_j \subset B_j$. Symmetrizing we have $ebe \in A^G \subset A$; in fact $ebe \in A_j$. Hence, if $\nu_\mathcal{F}(i) = k$ then $1 \in A_k ebe A_k$, and hence $1 \in B_k b B_k $. So we are done.
\end{proof}

\begin{proposition}\label{heart}
We have
$fdim A^G  \leq fdim A$. Moreover, if $fdim \, A = 1$ and there exists a finitely generated $A^G$-module $M$ such that $GK (M) \leq \frac{1}{2}GK (A^G)$, then $fdim \, A= fdim \, A^G$.
\end{proposition}

\begin{proof} Since  $A^G$ and $A*G$ are Morita equivalent by Theorem \ref{ring} (2), then 
the first claim follows from the above lemma and Theorem \ref{morita}. If $fdim \, A = 1$ then $fdim \, A^G \leq 1$. If there exists  a module $M$ which satisfying the hypothesis then $fdim \, A^G \geq 1$ by \cite[Corollary 1.7(i)]{B3}.  The statement follows.
\end{proof}

\

 \section{Holonomic modules for invariant subalgebras}\label{sec-hol-inv}
This section contains our main results. We assume that $\k$ has characteristic $0$ and use
the theory developed in previous sections to show that $A(\k)^G$ and $\mathcal{D}(\mathbb{A}^n/G)$ have a good theory of holonomic modules  for suitable actions of $G$ (Theorems \ref{Bernsteinnew}, \ref{quotient1}, \ref{goodcategory2}).
 The same holds for adequate invariant subrings of generalized Weyl algebras of pure type (Theorems \ref{genWeyl}, \ref{goodcategory3}).


\subsection{Invariant differential operators}
Let $A$ be a commutative $\mathsf{k}$-algebra. 
The algebra of differential operators on $A$ is defined as follows.
Set  $\mathcal{D}(A)_0=A$, $\mathcal{D}(A)_n = \{ d \in End_\mathsf{k} \, A | [d,a] \in \mathcal{D}(A)_{n-1}, \forall a \in A \}$, and  $\mathcal{D}(A)=\bigcup_{i=0}^\infty \mathcal{D}(A)_i$. This way we obtain a natural structure of filtered associative $\mathsf{k}$-algebra.

We recall 

\begin{proposition}\label{nsprop0}\cite[Chapter 15]{McConnell}
If $A$ is affine and regular then $\mathcal{D}(A)$ coincides with the subring of $End_\mathsf{k} \, A$ generated by $A$ and the module $Der_\mathsf{k} \, A$ of $\mathsf{k}$-derivations. $\mathcal{D}(A)$ is a simple affine Noetherian domain, and $GK (\mathcal{D}(A)) = 2 \, GK (A)$.
\end{proposition}

Through the rest of this subsection we assume that $A$ is an affine regular commutative domain with finite Krull dimension, i.e. the algebra of regular functions on a smooth affine irreducible variety. Then the algebra $\mathcal{D}(A)$ has a finite-dimensional filtration $\mathcal{F}$ such that $gr_\mathcal{F}\mathcal{D}(A)$ is affine commutative by \cite[15.1.21, 15.3.7,15.5.6]{McConnell},  that is $\mathcal{D}(A)$ is a somewhat commutative algebra (cf. Example \ref{example21}). In particular, $\mathcal{D}(A)$ is a simple algebra with multiplicity and $fdim \, \mathcal{D}(A) = 1$ (cf. Example \ref{example}). Note that $\mathcal{D}(A)$ is not an Artinian ring.

Let $G$ be a finite group of algebra automorphisms of $A$. Then $G$ acts on $\mathcal{D}(A)$ by conjugation:
 if $g \in G$ and $ d \in \mathcal{D}(A)$ then $g.d=gdg^{-1}$. The subalgebra $\mathcal{D}(A)^G$ of $G$-invariant differential operators 
 inherits a finite-dimensional filtration $\mathcal{F}'=\{A_i':=A^G \cap A_i \}_{i \geq 0}$.

\begin{proposition} \label{nsprop1}
$\mathcal{D}(A)^G$ is a somewhat commutative algebra with the filtration $\mathcal{F}'$.
\begin{proof} Since $A$ and $Der_\k \, A$ are $G$-stable then
without loss of generality one can assume that for each $A_i, i \geq 0$ in the filtration $\mathcal{F}$ holds $G(A_i) \subset A_i, i \geq 0$. This is due to \cite[Proposition 8.6.7, 8.6.9]{McConnell} and to the fact that $\mathcal{D}(A)$ is an almost centralizing extension of $A$ \cite[Theorem 15.1.20(i)]{McConnell}. Then we have that $gr_{\mathcal{F}'} \mathcal{D}(A)^G \simeq (gr_\mathcal{F} \mathcal{D}(A))^G$ by \cite[3.2.3]{Dumas}. By the Noether's Theorem, the right hand side of this isomorphism is an affine algebra and the statement follows.
\end{proof}
\end{proposition}

\begin{proposition} \label{nsprop2}
The units of $\mathcal D(A)$ are the units of $A$.
\end{proposition}

\begin{proof}
Let $y_1, \ldots, y_t$ be a transcendence basis of the field of fractions of  $A$. Suppose that $x \in \mathcal D(A)$ is a unit. A localization $\mathcal D(A)_c$ of $\mathcal D(A)$ by a certain non-zero regular element $c$ is isomorphic to an iterated Ore extension $ A[x_1; -\partial_{y_1}], \ldots, [x_t; -\partial_{y_t}]$ of $A$, by \cite[15.1.25, 15.2.6, 15.3.2]{McConnell}. 
Since  $\mathcal D(A)$ embeds into $\mathcal D(A)_c$, then $x$ is a unit of $ A[x_1; -\partial_{y_1}], \ldots, [x_t; -\partial_{y_t}]$, and hence of  $A$. 
\end{proof}



From Proposition  \ref{nsprop2} we immediately have

\begin{corollary}\label{nscorol1}
Let $G$ be a non-trivial finite group of automorphisms of $A$ with induced action on $\mathcal D(A)$ by conjugation. Then the action of $G$ is outer.
\end{corollary}

\begin{theorem}\label{nsthm1}
Let $G$ be a finite group of automorphisms of $A$ with induced action on $\mathcal D(A)$ by conjugation. Set $n = GK (A)$. Then
\begin{enumerate}
\item
$\mathcal D(A)^G$ is a simple ring Morita equivalent to $\mathcal D(A)*G$, both have the Gelfand-Kirillov dimension $2n$.
\item
$\mathcal D(A)^G$ is finitely generated Noetherian but not Artinian.
\item $fdim \, \mathcal D(A)^G = 1$, $h_{\mathcal D(A)^G}=n$, and  every finitely generated $\mathcal D(A)^G$ module $M$ satisfies the Bernstein inequality $GK (M) \geq n$.
\item Let $M\in \mathcal{H}(\mathcal D(A)^G)$ be a holonomic $\mathcal D(A)^G$-module. Then $M$ is  cyclic  torsion module with finite length bounded by $e(M)$. Moreover, $M$ is simple 
if $e(M)=1$.
\end{enumerate}
\end{theorem}

\begin{proof} Applying Theorem \ref{ring}, Proposition \ref{nsprop0} and Corollary \ref{nscorol1} we obtain
first two statements. 
Recall that $fdim \, \mathcal D(A)=1$. 
We have that $A^G$ is a $\mathcal D(A)^G$-module of the Gelfand-Kirillov dimension $n$. Hence,  by Proposition \ref{heart}, $fdim \, \mathcal D(A)^G = 1$. The Bernstein inequality then follows from Theorem \ref{BavBer}. Since we have explicitly constructed a module with the minimal possible Gelfand-Kirillov dimension, then $h_{\mathcal D(A)^G}=n$.
\end{proof}

\subsection{Invariants of the Weyl algebra}
Let $x_1, \ldots, x_n, y_1, \ldots, y_n$ be the standard generators of the Weyl algebra $A_n(\k)$, 
 identified with the algebra of $ \mathcal{D}(\k[x_1,\ldots,x_n])$ of differential operators on the polynomial ring. 
  The linear actions of  finite groups of automorphisms on $A_n(\k)$ by conjugation are   induced from the linear actions on $\k[x_1,\ldots,x_n]$. Recall that an automorphism of $A_n(\k)$ that fixes the subspace $span \langle x_1, \ldots, x_n, y_1, \ldots, y_n \rangle$ is called a \emph{symplectic automorphism}.  In particular, every  linear automorphism of $A_n(\k)$ is a
symplectic automorphism \cite{Dumas}.

\begin{corollary}\label{Weyl} Let $G$ be a finite group of symplectic automorphisms of $A_n(\k)$. We have:
\begin{enumerate}
\item
$A_n(\k)^G$ is a simple ring Morita equivalent to $A_n(\k)*G$, both have the Gelfand-Kirillov dimension $2n$.
\item
$A_n(\k)^G$ is finitely generated Noetherian but not Artinian.
\item
$A_n(\k)^G$ is a somewhat commutative algebra.
\end{enumerate}
\begin{proof}
Statements (1) and (2) follow immediately from Theorem \ref{ring}. Introduce a filtration $\mathcal{E}=\{E_i \}_{i \geq 0}$ on $A_n(\k)^G$ induced from the Bernstein filtration $\mathcal{B}=\{ B_i \}_{i \geq 0}$: $E_i = B_i \cap A_n(\k)^G$. Then $gr_\mathcal{E} A_n(\k)^G \simeq (gr_\mathcal{B} \, A_n(\k))^G$ (\cite{Dumas}), which is a finitely generated commutative algebra by the Noether's theorem. Statement (3) follows. \end{proof}
\end{corollary}

\begin{theorem}\label{Bernsteinnew} Let $G$ be a finite group of symplectic automorphisms of $A_n(\k)$. 
\begin{itemize}
\item $fdim \, A_n(\k)^G = 1$, $h_{A_n(\k)^G}=n$, and  every finitely generated $A_n(\k)^G$ module $M$ satisfies the Bernstein inequality $GK (M) \geq n$.
\item  If $M$ is a  holonomic $A_n(\k)$-module then $M$ is a cyclic torsion module of finite length bounded by $e(M)$. In particular, $M$ is simple
if $e(M)=1$.
\end{itemize}
\end{theorem}

\begin{proof} If the action of $G$ is linear then the statements follow from Theorem \ref{nsthm1}. We now consider an arbitrary action of $G$. First note that
 $G$ preserves the Bernstein filtration of $A_n(\k)$. Since $G$ is finite then we have $h_{A_n(\k)*G}=h_{A_n(\k)}=n$. Hence $h_{A^G}=n$ by Theorem \ref{morita}. Since the algebra  $A_n(\k)^G$ is somewhat commutative, there  exists a module $0 \neq M$ with the minimal Gelfand-Kirillov dimension. Hence,  $fdim \, A_n(\k)^G =1$ by Proposition \ref{heart}.
\end{proof}

\

\subsection{ Differential operators on quotient varieties}
From now on we assume that $\k$ is algebraically closed. Consider the quotient of $\mathbb{A}^n = Spec \, \k[x_1,\ldots,x_n]$ by $G$, and the ring of differential operators on $\mathbb{A}^n/G$.
Note that in general $\mathbb{A}^n/G$ is a singular variety, except when $G$ is a pseudo-reflection group by the Chevalley-Shephard-Todd theorem (\cite[Theorem 7.2.1]{Benson}).

Let $W$ be a finite group of linear automorphisms of $\k[x_1,\ldots,x_n]$, $N\subset W$ a subgroup generated by the pseudoreflections. Then $N$ is normal in $W$ (cf. \cite[pp. 259]{Traves}). Further on, we have 
 natural isomorphisms $$\k[x_1,\ldots,x_n]^W \simeq (\k[x_1, \ldots, x_n]^N)^{W/N} \simeq \k[x_1, \ldots, x_n]^{W/N},$$
 where the second isomorphism follows from the Chevalley-Shephard-Todd theorem.  Moreover, 
 the induced action of $W/N$ on the polynomial algebra is  linear. Now we have



\begin{lemma}\label{quotient1}
Let $W$ be any finite group of linear automorphisms of the polynomial ring $\k[x_1,\ldots,x_n]$ (and hence of the affine space $\mathbb{A}^n$). Then $\mathcal{D}(\mathbb{A}^n/W)$ is isomorphic to the ring of invariants of $A_n(\k)$ under the action of a finite group of linear automorphisms.
\end{lemma}

\begin{proof}
 Since $W/N$ does not contain  pseudoreflections then using the isomorphisms above and \cite[Theorem 5]{Levasseur} we have an isomorphism of  algebras of invariant differential operators $$\mathcal{D}(\k[x_1,\ldots,x_n]^N)^{W/N}  \simeq \mathcal{D}(\k[x_1, \ldots, x_n]^W).$$  Since $W/N$  acts linearly on $$\k[x_1,\ldots,x_n]^N \simeq \k[x_1,\ldots,x_n],$$ we have $\mathcal{D}(\mathbb{A}^n/W) \simeq \mathcal{D}(\mathbb{A}^n)^{W'}$
 with linear action of $W'\simeq W/N$.
\end{proof}

From Lemma \ref{quotient1} and Theorem \ref{Bernsteinnew} we immediately have

\begin{theorem}\label{goodcategory2} Let $W$ be a finite group of linear automorphisms of $\mathbb{A}^n$.
\begin{itemize}
\item
The ring  $\mathcal{D}(\mathbb{A}^n/W)$ has the filter dimension $1$;
\item If  $M$ is a holonomic $\mathcal{D}(\mathbb{A}^n/W)$-module then $M$ is  cyclic  torsion module with finite length bounded by $e(M)$. In particular, 
if $e(M)=1$  then $M$ is simple.
\end{itemize}
\end{theorem}

\

\subsection{Invariants of generalized Weyl algebras of pure type}
 We keep the hypothesis that $\k$ is algebraically closed.
Let us recall the definition of the Shephard-Todd groups of type $G(m,p,n)$.
Let $G_m \subset \k $ be the cyclic group in $m$ elements, generated by the $m$-th roots of unity. Let $A(m,p,n)$ be the subgroup of $G_m^{\otimes n}$ consisting of $(h_1, \ldots, h_n)$ such that $(\prod h_i)^{m/p}=1$. Set $G(m,p,n)= A(m,p,n) \rtimes S_n$.  It is always a normal subgroup of $G(m,1,n)$, and the quotient group is isomorphic to $G_p$. The groups of type $G(m,p,n)$ are the non-exceptional irreducible complex reflection groups in the classification of Shephard-Todd.

The following is clear.

\begin{proposition} \label{action}
Let $A=D(a, \sigma)$ be a generalized Weyl algebra of rank $n$ of pure type, so that $D=  k[h_1,\ldots,h_n]$ or $\k[h_1^\pm,\ldots,h_n^\pm]$. Then $A$ is equipped with the following natural action of $G(m,p,n)$:
 if $\xi = (g, \pi) \in A(m,p,n) \rtimes S_n$, $g=(g_1, \ldots, g_n)$, then $\xi(h_i) =  h_{\pi(i)}$, $\xi X_i =g_i X_{\pi(i)}$, $\xi Y_i = g_i^{-1} Y_{\pi(i)}$. 
\end{proposition}

We have

\begin{theorem}\label{genWeyl}
Let $A=D(a, \sigma)$ be a generalized Weyl algebra of rank $n$ of pure type, and $G=G(m,p,n)$. Then we have:
\begin{itemize}
\item
$A^G$ is a finitely generated simple Noetherian ring which is not Artinian;  $A^G$  is Morita equivalent to $A*G$.
\item
For every finitely generated $A^G$-module $M$, $GK (M) \geq n$.
\item
$fdim \, A^G = 1$.
\end{itemize}
\end{theorem}

\begin{proof} Since $G=G(m,p,n)$ acts by outer automorphisms on $A$, we are in the position to apply Theorem \ref{ring} and Proposition \ref{heart}.
Item (1) follows by the same proof as of Corollary \ref{Weyl}. Since $G$ preserves the filtration $\mathcal{B}_A$, it is clear  that $h_{A*G}=h_A=n$. Hence the statement (2) follows from Theorem \ref{morita}. Finally, Statement (3) follows from Proposition \ref{heart}.
\end{proof}

Finally, consider the category of holonomic modules $\mathcal{H}(A)$ for $A=D(a, \sigma)^{G(m,p,n)}$. 

\begin{theorem}\label{goodcategory3}
Let $A=D(a, \sigma)^{G(m,p,n)}$ be the invariant subalgebra of a generalized Weyl algebra of rank $n$ of pure type. 
 If $M \in \mathcal{H}(A)$ then $M$ is a cyclic torsion module with finite length bounded by $e(M)$. If 
 $e(M)=1$  then $M$ is simple.
\end{theorem}

Let us now revisit Theorem \ref{curious}. We have

\begin{theorem}\label{curious3}
Let $A$ be one of the following algebras:
\begin{enumerate}
\item
 Any  subalgebra of invariants of $A_1(\mathbb{C})$ under the action of a finite group;
\item
The ring of differential operators $\mathcal{D}(X)$, where $X$ is an smooth affine curve;
\item
A generalized Weyl algebra of rank $1$ of simple classical or quantum type.
\end{enumerate}
 If $M$ is a finitely generated $A$-module then $GK (M) =1$ if and only if $M$ has a torsion.
\end{theorem}

\begin{proof} Let $G$ be any finite group of automorphisms of $A_1(\mathbb{C})$. Then $G$ is conjugated in $Aut_\mathbb{C} A_1(\mathbb{C})$ to a group of symplectic automorphisms \cite{AD} (cf. \cite{Dumas}).
Now the statement in the case of the first Weyl algebra follows from Theorem \ref{curious} and Theorem \ref{Bernsteinnew}. The case of differential operators on curves follows from Theorem \ref{curious} and Example \ref{example21}. The case of generalized Weyl algebras follows from Theorem \ref{curious} and Theorem \ref{GWA}.
\end{proof}

\section{Rational Cherednik algebras}\label{section-cherednik}
\subsubsection{Generalized somewhat commutative algebras}

Let $W$ be a complex reflection group acting on a complex vector space $H$, $S$  the set of reflections.
 For each $s \in S$, take $\alpha_s \in H^*$ and $\alpha_s^\vee \in H$ such that $ \alpha_s$ is an eigenvector of $\lambda_s$ (the non-trivial eigenvalue of $s$ in $h^*$); and $\alpha_s  ^\vee$ is an eigenvector of $\lambda_s^{-1}$ (the non-trivial eigenvalue of $s$ in $H$). Normalize them in such a way  that  $(\alpha_s, \alpha_s^\vee)=2$ with respect to the natural pairing $H^* \times H \rightarrow \C$. Finally, let $c: S \rightarrow \mathbb{C}$ be the  invariant  conjugation function. 

 The {\it rational Cherednik algebra} $H_{c,t}(W,H)$, $t \in \C$ is the quotient of $\mathbb{C}W \ltimes T(H \oplus H^*)$ by the relations:

\[ [x,x']=[y,y']=0; \, [y,x]=tx(y)-\sum_{s \in S} c(s)(y, \alpha_s) (x, \alpha_s^\vee) s, \]
with $x,x' \in H^*, y,y' \in H$.

We also consider the {\it spherical subalgebra} $U_{c,t}(W,H):= e H_{c,t}(W,H) e$ of $H_{c,t}(W,H)$, where $e:= \frac{1}{|W|} \sum_{w \in W} w$.

Without loss of generality we assume $t=1$ and for simplicity  just write $H_c$ and $U_c$ for the algebras above, and $\mathsf{n}$ for $dim \, H$. Recall that for a generic $c$, both algebras are Morita equivalent, simple Noetherian but not Artinian rings  (cf. \cite{Thompson}).

\

Let us call an algebra $A$ \emph{generalized somewhat commutative algebra} if it has a finite filtration $\mathcal{F}=\{A_i \}_{i \geq 0}$ with $\k \subset A_0$ and $dim \, A_i < \infty, i \geq 0$, such that the associated graded algebra is affine commutative.
 The difference between this definition and the definition of somewhat commutative algebra is that we do not impose a condition $A_0 = \k$. Nonetheless, we have:

\begin{theorem}  Let $A$ is an affine simple generalized somewhat commutative algebra. Then 
$A$ is an algebra with multiplicity.
\end{theorem}

\begin{proof}
It follows from \cite[Theorem 3.2, Proposition 3.3]{B1}, using the fact the dimension defined there equals to the Gelfand-Kirillov dimension by \cite[Lemma 2.1, Proposition 6.6]{Krause}.
\end{proof}

\begin{corollary}\label{cherednik1}
For a generic  $c$, the spherical subalgebra $U_c$ is a generalized somewhat commutative algebra with multiplicity. Moreover,  $GK (H_c) = GK (U_c) = 2 \mathsf{n}$.
\end{corollary}

\begin{proof}
 The first statement follows from \cite[1.6]{Bellamy}, as it shows that for an adequate finite filtration, $gr \, U_c \simeq \C[h \oplus h^*]^W$. Then the second statement follows from  \cite[Proposition 6.6]{Krause} and \cite[Propostion 8.2.9(iii)]{McConnell}.
\end{proof}

\

\subsection{Filter and Krull dimensions of rational Cherednik algebras}
 We assume that $c$ is a generic parameter.

\begin{lemma}\label{cherednik11}
$fdim \,  H_c = fdim \, U_c \geq 1$.
\end{lemma}

\begin{proof}
The Gelfand-Kirillov  dimension the polynomial representation of $H_c$ equals  $\frac{1}{2} GK (H_c)$. Hence, $fdim \, H_c \geq 1$ by \cite[Corollary 1.7(i)]{B3}. Since $fdim$ is Morita invariant,  we conclude that $fdim \, U_c \geq 1$.
\end{proof}

\begin{corollary}\label{cherednik3}
If $M \in \mathcal{H}(U_c)$ then $M$ is a cyclic torsion module with finite length bounded by $e(M)$. Moreover, if 
 $e(M)=1$  then $M$ is simple.
 \end{corollary}

\begin{proof}
Follows from Lemma \ref{cherednik11}, Theorem \ref{BavBer} and Theorem \ref{cherednik2}.
\end{proof}

We have

\begin{theorem}\label{cherednik2}
$fdim \, H_c = fdim \, U_c = 1$.
\end{theorem}

\begin{proof}
The algebra $U_c$ is an algebra with multiplicity by Corollary \ref{cherednik1}. Let $h$ be the holonomic number of $U_c$. Then there exists a finitely generated $U_c$-module $M$ such that $GK (M) =h$. Since the Bernstein's inequality  holds for $M$ (cf. \cite[Proposition 3.7]{Thompson}), $h = \frac{1}{2} GK (U_c)$. Hence  $fdim U_c \leq 1$ by Theorem \ref{BavBer}. Since $fdim$ is a Morita invariant by Theorem \ref{morita}, then we have  $fdim \, U_c = fdim \, H_c=1$ by Lemma \ref{cherednik11}.
\end{proof}

Next we compute the Krull dimension $\mathcal{K}$ (in the sense of Gabriel-Rentschler, cf. \cite[Chapter 6]{McConnell}) of these algebras.

\begin{theorem}\label{cherednik4}
$\mathcal{K}(H_c)= \mathcal{K}(U_c) = \mathsf {n}$.
\end{theorem}

\begin{proof}
Since the Krull dimension is a Morita invariant (\cite[Proposition 6.5.1]{McConnell}), it suffices to show it for $U_c$. By the Dunkl embedding we have that a localization of the spherical subalgebra is isomorphic to $D(h_{reg}/W)$ (\cite[Proposition 4.4.1]{Bellamy}). By \cite[Lemma 6.5.3(iib)]{McConnell}, $\mathcal{K}(U_c) \geq \mathcal{K}(D(h_{reg}/W))$, and the later equals  $\mathsf{n}$ by \cite[Theorem 15.3.7]{McConnell}. Since $fdim \, U_c =1$ and  $GK (U_c) =2 \mathsf{n}$, by \cite[Theorem 1.3]{Bnew1}, then we  have $\mathcal{K}(U_c) \leq \mathsf{n}$. Hence the equality follows.
\end{proof}

\section*{Acknowledgments}

V.\,F.\ is supported in part by the CNPq (304467/2017-0) and by the Fapesp (2018/23690-6); J.S. is supported by the Fapesp (2018/18146-5).

\

\end{document}